\documentclass[12pt]{amsart}
\usepackage{bm}
\usepackage{amscd,amsmath,amsthm,amssymb}
\usepackage{pstcol,pst-plot,pst-3d}%\usepackage[T1]{fontenc}
\usepackage{color}
\usepackage{pstricks}
\usepackage{lineno,xcolor}
\usepackage{tikz}
\DeclareSymbolFont{largesymbol}{OMX}{yhex}{m}{n}
\DeclareMathAccent{\Widehat}{\mathord}{largesymbol}{"62}

\UseRawInputEncoding % ??utf-8??????

\newpsstyle{fatline}{linewidth=1.5pt}
\newpsstyle{fyp}{fillstyle=solid,fillcolor=verylight}
\definecolor{verylight}{gray}{0.97}
\definecolor{light}{gray}{0.9}
\definecolor{medium}{gray}{0.85}
\definecolor{dark}{gray}{0.6}

 \def\G{{\mathcal G}}

 \def\opn#1#2{\def#1{\operatorname{#2}}} % to make operators
 \opn\chara{char} \opn\length{\ell} \opn\pd{pd} \opn\rk{rk}
 \opn\projdim{proj\,dim} \opn\injdim{inj\,dim} \opn\rank{rank}
 \opn\depth{depth} \opn\grade{grade} \opn\height{height}
 \opn\embdim{emb\,dim} \opn\codim{codim}
 
 \opn\Tr{Tr} \opn\bigrank{big\,rank}
 \opn\superheight{superheight}\opn\lcm{lcm}
 \opn\trdeg{tr\,deg}%\emph{
 \opn\reg{reg} \opn\lreg{lreg} \opn\ini{in} \opn\lpd{lpd}
 \opn\size{size} \opn\sdepth{sdepth}
 \opn\link{link}\opn\fdepth{fdepth}\opn\lex{lex}
 \opn\tr{tr}
 \opn\type{type}
 \opn\Borel{Borel}
\opn\cdeg{cdeg}

 \opn\div{div} \opn\Div{Div} \opn\cl{cl} \opn\Cl{Cl}
 \opn\Spec{Spec} \opn\Supp{Supp} \opn\supp{supp} \opn\Sing{Sing}
 \opn\Ass{Ass} \opn\Min{Min}\opn\Mon{Mon}
 \opn\Ann{Ann} \opn\Rad{Rad} \opn\Soc{Soc}
 \opn\Im{Im} \opn\Ker{Ker} \opn\Coker{Coker} \opn\Am{Am}
 \opn\Hom{Hom} \opn\Tor{Tor} \opn\Ext{Ext} \opn\End{End}
 \opn\Aut{Aut} \opn\id{id}
 
 \opn\nat{nat}
 \opn\pff{pf}%   \pf exists already
 \opn\Pf{Pf} \opn\GL{GL} \opn\SL{SL} \opn\mod{mod} \opn\ord{ord}
 \opn\Gin{Gin} \opn\Hilb{Hilb}\opn\sort{sort}
 \opn\PF{PF}\opn\Ap{Ap}
 \opn\aff{aff} \opn
 \con{conv} \opn\relint{relint} \opn\st{st}
 \opn\lk{lk} \opn\cn{cn} \opn\core{core} \opn\vol{vol}  \opn\inp{inp} \opn\nilpot{nilpot}
 \opn\link{link} \opn\star{star}\opn\lex{lex}\opn\set{set}
 \opn\width{wd}
 \opn\Fr{F}
 \opn\QF{QF}
 \opn\G{G}
 \opn\type{type}\opn\res{res}
 \opn\gr{gr}
 
  \def\cdeg{deg}

 \def\pot#1#2{#1[\kern-0.28ex[#2]\kern-0.28ex]}

 \opn\dirlim{\underrightarrow{\lim}}
 \opn\inivlim{\underleftarrow{\lim}}
 \def\Implies{\ifmmode\Longrightarrow \else
         \unskip${}\Longrightarrow{}$\ignorespaces\fi}
 \def\implies{\ifmmode\Rightarrow \else
         \unskip${}\Rightarrow{}$\ignorespaces\fi}
 \def\iff{\ifmmode\Longleftrightarrow \else
         \unskip${}\Longleftrightarrow{}$\ignorespaces\fi}

 \let\:=\colon
 \newtheorem{Theorem}{Theorem}[section]
 \newtheorem{Lemma}[Theorem]{Lemma}
 \newtheorem{Corollary}[Theorem]{Corollary}
 
 \newtheorem{Remark}[Theorem]{Remark}
 
 \newtheorem{Example}[Theorem]{Example}
 
 \newtheorem{Definition}[Theorem]{Definition}

 \let\epsilon\varepsilon
 \let\kappa=\varkappa
 \def\qed{\ifhmode\textqed\fi
       \ifmmode\ifinner\quad\qedsymbol\else\dispqed\fi\fi}
 \def\textqed{\unskip\nobreak\penalty50
        \hskip2em\hbox{}\nobreak\hfil\qedsymbol
        \parfillskip=0pt \finalhyphendemerits=0}
 \def\dispqed{\rlap{\qquad\qedsymbol}}
 
 \opn\dis{dis}
 \def\pnt{{\raise0.5mm\hbox{\large\bf.}}}
 
 \opn\Lex{Lex}

\begin{document}

\title {Freiman cover ideals of unmixed bipartite graphs }

\author {Guangjun Zhu$^{^*}$,  Yakun Zhao and  Yijun Cui }

\address{Authors¡¯ address:  School of Mathematical Sciences, Soochow University, Suzhou 215006, P. R. China}
\email{zhuguangjun@suda.edu.cn(Corresponding author:Guangjun Zhu),
\linebreak[4]1768868280@qq.com(Yakun Zhao),237546805@qq.com(Yijun Cui).}

\dedicatory{ }

\begin{abstract}
An equigenerated monomial ideal $I$ in the polynomial ring $R=k[z_1,\ldots, z_n]$ is a
Freiman ideal if $\mu(I^2)=\ell(I)\mu(I)-{\ell(I)\choose 2}$ where $\ell(I)$ is the analytic spread of $I$ and $\mu(I)$ is the number of minimal generators of $I$.  In this paper we classify all simple  connected unmixed bipartite graphs  whose cover ideals  are Freiman ideals.
\end{abstract}

\thanks{* Corresponding author}



%		13H10   	Special types (Cohen-Macaulay, Gorenstein, Buchsbaum, etc.)
%		13D02   	Syzygies, resolutions, complexes
%		05E40   	Combinatorial aspects of commutative algebra
%		16S36   	Ordinary and skew polynomial rings and semigroup rings

%		14M25   	Toric varieties, Newton polyhedra [See also 52B20]
%		13A02   	Graded rings
%		13F20   	Polynomial rings and ideals; rings of integer-valued polynomials
%		13A18   	Valuations and their generalizations
%		06A11   	Algebraic aspects of posets

 \maketitle

 \setcounter{tocdepth}{1}

 \section*{Introduction}
Let $I$ be a monomial ideal in the polynomial ring $R=k[z_1,\ldots,z_n]$ over a field $k$. Let $\mu(I)$
denote the least number of generators of $I$.
It is  a very difficult problem to exactly compute  $\mu(I^k)$ for each integer $k\geq 2$.
If $I$ is generated by a regular sequence and $\mu(I)=m$, it is well known that $\mu(I^k)={k+m-1\choose m-1}$,
which is the maximal that $\mu(I^k)$ can reach.
At the other extreme, Eliahou et. al. \cite{EHS} constructed ideals of height $2$ in $k[x,y]$ such that
 $\mu(I^2)=9$ and $\mu(I)=m$, where $m\geq 6$ can be any integer.

As a consequence of a well known theorem from additive number theory, due to Freiman \cite{Fre}, Herzog et. al. showed in  \cite[Theorem 1.8]{HMZ} that if $I$ is an equigenerated monomial ideal, that is, all its generators are of the same degree,
then $\mu(I^2)\geq \ell(I)\mu(I)-{\ell(I)\choose 2}$,
 where $\ell(I)$ is the analytic spread of $I$. If the equality holds, then Herzog and Zhu \cite{HZ1} called this ideal $I$ to be a {\em Freiman} ideal (or simply Freiman).
Herzog et. al. provided in \cite[Theorem 2.3]{HHZ} several equivalent conditions for an equigenerated monomial ideal $I$ to be a Freiman ideal.
For example, they show that $I$ is Freiman if and only if $\mu(I^k)={\ell(I)+k-2\choose k-1}\mu(I)-(k-1){\ell(I)+k-2\choose k}$
for some $k\geq 2$ if and only if the fiber
cone $F(I)$ of $I$   has minimal multiplicity if and only if  $F(I)$  is Cohen-Macaulay  and its defining ideal has a $2$-linear resolution.
It is a very restrictive condition for ideals arising from combinatorial structures, which often guarantees strong combinatorial properties.

Freiman ideals in the classes of principal Borel ideals, Hibi ideals, Veronese type ideals, matroid ideals,  sortable ideals, edge ideals of
graphs, and  cover ideals of  some classes of graphs such as trees, circulant graphs, and whiskered graphs have been studied (see \cite{DG,HHZ,HZ1,HZ2}).

The  purpose of this paper is to  give a full classification of Freiman ideals for cover ideals of all connected unmixed bipartite graphs.

Let $[n]$ denote the set $\{1,\ldots,n\}$. Thanks to this fact, let  $G$ be  an unmixed bipartite graph on the vertex set $V(G)=\{x_1,\dots,x_n\}\cup \{y_1,\dots, y_n\}$, then each of its minimal vertex covers has the form $\{x_{i_1},\ldots,x_{i_s},y_{i_{s+1}}, \ldots,y_{i_n}\}$, where $\{i_1,\ldots, i_n\}=[n]$.   Therefore, the set
$$\mathcal{L}_G=\{C\cap \{x_1,\ldots,x_n\} \mid C \text{\ is a minimal vertex cover of \ }G\}$$
is a sublattice of the Boolean
lattice $\mathcal{L}_n$ on $P_n=\{p_1,\ldots,p_n\}$ such that $\emptyset\in \mathcal{L}_G$ and $P_n\in \mathcal{L}_G$.
In section 2, we associated the lattice $\mathcal{L}_G$ with its Hibi ideal $H_{\mathcal{L}_G}$,
 which is exactly the cover ideal of $G$, and show that  the result  as follows

\begin{Theorem}
Let  $G$ be a  connected unmixed bipartite graph on the vertex set $V_n=\{x_1,\dots, x_n\}\cup \{y_1,\dots, y_n\}$, then  $H_{\mathcal{L}_{G}}$ is Freiman  if and only if  $G$ is an almost  complete C-M bipartite graph,   or
there exist some  complete bipartite subgraphs $K_{m_{1},m_{1}},\ldots,K_{m_{s},m_{s}}$ of $G$ such that the induced graph of $G$ on
$V_n\setminus (\bigcup\limits_{i=1}^s V_{m_i-1})$ is an almost  complete C-M bipartite graph, where
 $V_{m_i}=\{x_{i1},\ldots, x_{im_i}\}\cup \{y_{i1},\ldots, y_{im_i}\}$ is  the vertex set of $K_{m_{i},m_{i}}$ for any $i\in [s]$.
\end{Theorem}

\medskip
We greatfully acknowledge the use of the computer algebra system CoCoA (\cite{Co}) for our experiments.

\medskip

\section{Preliminaries}
In this section, we gather together the needed definitions and basic facts, which
will be used throughout this paper. However, for more details, we refer the reader
to \cite{BM,HH1,HZ1,HHZ,Va}.

Let $G$  be a simple (i.e., finite, undirected, loopless and without multiple
edges) graph with the vertex set $V(G)$ and the edge set $E(G)$. Throughout this paper, we assume that $G$ is connected without isolated vertices. A vertex cover of
$G$ is a subset $C\subset V(G)$ such that  each edge has at least one
vertex in $C$.  Such a vertex cover $C$ is called minimal if no subset $C'\subsetneq C$ is a vertex cover
of $G$. The  graph $G$ is unmixed   (also called well-covered)  if all minimal vertex covers of $G$ have
the same cardinality.  The graph  $G$  is called  bipartite if there are two disjoint subsets $W,W'\subset V(G)$ such that $V(G)=W\cup W'$, and $E(G)\subseteq W\times W'$.

As stated before, if $G$ is an unmixed bipartite graph on the vertex set $V(G)=\{x_1,\dots,x_n\}\cup \{y_1,\dots, y_n\}$, then  each of its minimal vertex covers has the form $\{x_{i_1},\ldots,x_{i_s},y_{i_{s+1}}, \ldots,y_{i_n}\}$, where $\{i_1,\ldots, i_n\}=[n]$.  It is shown in \cite{HHO} that the set
$$\mathcal{L}_G=\{C\cap \{x_1,\ldots,x_n\} \mid C \text{\ is a minimal vertex cover of \ }G\}$$
is a sublattice of the Boolean
lattice $\mathcal{L}_n$ on $P_n=\{p_1,\ldots,p_n\}$ such that $\emptyset\in \mathcal{L}_G$ and $P_n\in \mathcal{L}_G$, and for any  sublattice $\mathcal{L}$ of $\mathcal{L}_n$ such that $\emptyset\in \mathcal{L}$ and $P_n\in \mathcal{L}$, there exists an unmixed bipartite graph $G$ such that $\mathcal{L}=\mathcal{L}_G$.
Attached to this lattice $\mathcal{L}_G$ is a monomial ideal $H_{\mathcal{L}_G}$ in the polynomial ring
$S=k[x_1,\dots,x_n,y_1,\dots,y_n]$, this ideal is called the {\em Hibi ideal} of $\mathcal{L}_G$.
To define $H_{\mathcal{L}_G}$, let $\mathcal{J}(P_n)$ be the set of poset ideals
of $P_n$. Recall that a subset $J\subset P_n$ is called a poset ideal of $P_n$, if for all $p\in J$ with $q\in  P_n$ and $q\leq p$, it
follows that $q\in J$. Now $H_{\mathcal{L}_G}$ is defined as follows:
$$
H_{\mathcal{L}_G}=(\{u_J\}_{J\in \mathcal{J}(P_n)}), \text{\ \ where\ \ } u_J=(\prod\limits_{p\in J}x_p)(\prod\limits_{p\in P_n\setminus J}y_p).
$$
The ideal
$$I(G)=({x_iy_j |\{x_i,y_j\} \in E(G)})$$
is called the edge ideal of $G$. Note that $H_{\mathcal{L}_G}$  is the Alexander dual of $I(G)$ and  actually the cover ideal of $G$.
Since each $u_J$ is a squarefree monomial of degree $n$, and $H_{\mathcal{L}_G}$  is  an equigenerated monomial ideal.

 We say that  a simple graph $G$ is Cohen-Macaulay (C-M for short)  if the quotient ring $R/I(G)$ is Cohen-Macaulay.  Observe that if $G$ is C-M then $G$ is unmixed by \cite{BH}. A graph theoretical characterization of C-M bipartite graphs was given in \cite{HH2} and that of unmixed bipartite graphs was given in \cite{Vi}.

\begin{Lemma}\label{CM}{\em(\cite[Theorem 3.4]{HH2})}
Let $G$ be a bipartite graph  with the bipartition of its vertices
$V(G)=\{x_1,\dots, x_m\}\cup \{y_1,\dots, y_n\}$. Then  $G$ is   C-M  if and only if $n=m$, and there is a labeling such that
\begin{itemize}
\item[(i)] $\{x_i,y_i\}\in  E(G)$ for all $i\in [n]$;
\item[(ii)] If $\{x_i,y_j\}\in  E(G)$, then $i\leq j$;
\item[(iii)] If $\{x_i,y_j\},\{x_j,y_k\}\in  E(G)$  with $i<j<k$, then $\{x_i,y_k\}\in  E(G)$.
\end{itemize}
\end{Lemma}

\begin{Lemma}\label{unmixed}{\em(\cite[Theorem 1.1]{Vi})}
Let $G$ be  a bipartite graph. Then $G$ is unmixed if and only if there is a bipartition $V(G)=\{x_1,\dots, x_n\}\cup \{y_1,\dots, y_n\}$ such that
\begin{itemize}
\item[(i)] $\{x_i,y_i\}\in  E(G)$ for all $i\in [n]$;
\item[(ii)] If $\{x_i,y_j\},\{x_j,y_k\}\in  E(G)$  for distinct $i,j,k\in [n]$, then $\{x_i,y_k\}\in  E(G)$.
\end{itemize}
\end{Lemma}

\begin{Definition} \label{exam}
A  bipartite  graph $G$ is called   complete C-M  if there exists a bipartition of $V(G)=\{x_1,\dots, x_n\}\cup \{y_1,\dots, y_n\}$ such that,  for any $i,j\in [n]$, $\{x_i,y_j\}\in  E(G)$ if and only if  $i\leq j$.
\end{Definition}

Let  $(P_n,\leq)$ be a finite  poset with $P_n=\{p_1,\ldots,p_n\}$. We write $G(P_n)$ for the bipartite graph on the
 set $V_n=W\cup W'$, where $W=\{x_1,\dots, x_n\}$ and $W'=\{y_1,\dots, y_n\}$, whose edges are those $\{x_i,y_j\}$ such that $p_i\leq p_j$.
A bipartite graph $G$ on $V_n=W\cup W'$  is said to come from a
poset, if there exists a finite poset $P_n$ on $\{p_1,\ldots,p_n\}$ such that $p_i\leq p_j$ implies
$\{x_i,y_j\}\in E(G)$, and after relabeling of the vertices of $G$ one has $G =G(P_n)$. Herzog and Hibi in \cite{HH2}
proved that a bipartite graph $G$ is C-M if and only if $G$ comes
from a poset.

\begin{Remark}
A  complete C-M bipartite graph $G$ is connected and a C-M bipartite graph, hence it is also an unmixed bipartite graph. It follows that $G$  comes from a poset.
\end{Remark}

A poset $P$ is called a {\em chain} if it is totally ordered, that is
if any two elements of $P$ are comparable.

\begin{Theorem}\label{C-Mcomplete}
Let $G$ be an unmixed  bipartite graph, which comes from a poset $P_n$. Then $G$ is a  complete  C-M  bipartite graph if and only if $P_n$ is a chain.
\end{Theorem}
\begin{proof} Let $(P_n, \leq)$ be a poset on $P_n=\{p_1,\ldots,p_n\}$ such that $p_i\leq p_j$
implies $i\leq j$, and $G=G(P_n)$.   For any $i,j\in [n]$ with $i\leq j$, we obtain  by Definition  \ref{exam} that $G$ is a complete  C-M  bipartite graph if and only if  $\{x_i,y_j\}\in E(G)$ if and only if  $p_i\leq p_j$, as  desired.
\end{proof}

\begin{Theorem}\label{zhao}
Let $G$ be a C-M bipartite graph on the  vertex set $V(G)=\{x_1,\dots, x_n\}\\
\cup \{y_1,\dots, y_n\}$, and   $H_{\mathcal{L}_G}$  the Hibi ideal of $G$. Then $H_{\mathcal{L}_G}$ is Freiman if and only if  the induced subgraph of $G$ on  $V(G)\setminus \{x_{i},y_{i}\}$ is a  complete C-M bipartite graph  for some $i\in[n]$.
\end{Theorem}
\begin{proof} Let $(P_n, \leq)$ be a poset on $P_n=\{p_1,\ldots,p_n\}$  and $G=G(P_n)$ comes from the poset $P_n$.
It follows from \cite[Theorem 4.4]{HZ1} that  $H_{\mathcal{L}_G}$ is Freiman if and only if there exist some $i\in [n]$ such that the subposet $P_{n}\setminus \{p_{i}\}$ is a chain. Let $H$ be the induced subgraph of $G$ on  $V(G)\setminus \{x_{i},y_{i}\}$. Then
$H$ comes from the subposet $P_n\setminus \{p_{i}\}$. Indeed, let $p_k,p_\ell\in P_n\setminus \{p_{i}\}$ such that $p_k\leq p_\ell$, it is obvious that $k,\ell\notin \{i\}$  and  $p_k\leq p_\ell$ in $P_n$. It implies that  $\{x_k,y_\ell\}\in E(G)$ and $k,\ell\notin \{i\}$. Hence  $\{x_k,y_\ell\}\in E(H)$.
The desired results from Theorem \ref{C-Mcomplete}.
\end{proof}

\medskip
A typical example of a Freiman  cover ideal of an unmixed bipartite graph is shown in Figure $1$.
\medskip
\begin{center}
\setlength{\unitlength}{1.5mm}
\begin{picture}(120,35)
\linethickness{1pt}
\thicklines
\multiput(5,10)(10,0){4}{\circle*{1}}
\multiput(5,30)(10,0){4}{\circle*{1}}
\multiput(5,10)(10,0){4}{\line(0,1){20}}
\multiput(5,30)(10,0){3}{\line(1,-2){10}}
\multiput(5,30)(10,0){2}{\line(1,-1){20}}
\put(5,30){\line(3,-2){30}}
\put(15,10){\line(1,2){10}}
\put(4,32){$x_{1}$}
\put(14,32){$x_{2}$}
\put(24,32){$x_{4}$}
\put(34,32){$x_{3}$}
\put(4,6){$y_{1}$}
\put(14,6){$y_{2}$}
\put(24,6){$y_{4}$}
\put(34,6){$y_{3}$}
\put(14,-2){$Figure\ 1$}

\multiput(65,10)(10,0){4}{\circle*{1}}
\multiput(65,30)(10,0){4}{\circle*{1}}
\multiput(65,10)(10,0){4}{\line(0,1){20}}
\multiput(65,30)(20,0){2}{\line(1,-2){10}}
\put(64,32){$x_{1}$}
\put(74,32){$x_{2}$}
\put(84,32){$x_{3}$}
\put(94,32){$x_{4}$}
\put(64,6){$y_{1}$}
\put(74,6){$y_{2}$}
\put(84,6){$y_{3}$}
\put(94,6){$y_{4}$}
\put(69,1){$G_1$}
\put(89,1){$G_2$}
\put(74,-3){$Figure\  4$}
\end{picture}
\end{center}

\vspace{0.5cm}
\medskip

\medskip

\section{unmixed bipartite graphs}

A  bipartite graph is called {\em complete}  if its two vertices are adjacent if and only if they are in different partite sets. The complete
bipartite graph with partite sets of size $m$ and $n$ is denoted $K_{m,n}$.

\begin{Theorem}\label{complete}
Let $G$ be  an unmixed bipartite graph on the vertex set $V_n=\{x_1,\dots, x_n\}\cup \{y_1,\dots, y_n\}$.
Then the following conditions are equivalent:
\begin{itemize}
\item[(i)] $G$ is C-M;
\item[(ii)] Any complete bipartite graph $K_{m,m}$ with the  vertex set $\{x_{i_{1}},\ldots,x_{i_{m}}\}\cup\{y_{i_{1}},\ldots,\\
y_{i_{m}}\}$ cannot be an induced  subgraph of $G$, where $m\geq 2$ and $\{i_1,\ldots,i_m\}\subset [n]$;
\item[(3)] Any complete bipartite graph $K_{2,2}$ with the  vertex set $\{x_{i},x_{j}\}\cup\{y_{i},y_{j}\}$ cannot be an induced  subgraph of $G$, where $\{i,j\}\subset [n]$.
\end{itemize}
\end{Theorem}
\begin{proof}
$(1)\Rightarrow(2)$ follows from (ii) of Lemma \ref{CM}, and   $(2)\Rightarrow(3)$ is trivial.

$(3)\Rightarrow(1)$. We only need to find a finite poset $P_n$ such that $G=G(P_n)$ comes from this poset. Let $P_n=\{p_{1},\ldots,p_{n}\}$ and $\leq$
 denote the binary relation on $P_n$ defined by setting $p_{i}\leq p_{j}$ if and only if $\{x_{i},y_{j}\}\in E(G)$.   Lemma \ref{unmixed} implies that $p_{i}\leq p_{i}$ for any $i\in [n]$, and  if $p_{i}\leq p_{j}$ and $p_{j}\leq p_{k}$ for any
$i, j, k\in [n]$, then $p_{i}\leq p_{k}$.
Let $p_{i}\leq p_{j}$ and  $p_{j}\leq p_{i}$ for any $i,j\in [n]$, then $i=j$. Otherwise,
the complete bipartite graph $K_{2,2}$ on  the vertex set $\{x_{i},x_{j}\}\cup\{y_{i},y_{j}\}$ is an induced  subgraph of $G$,  a contradiction.
Thus $\leq$ is a partial order on $P_n$ and
$G=G(P_n)$ comes from the poset $P_n$ by \cite[Lemma 3.1]{HH2}. Hence $G$ is C-M, as desired.
\end{proof}

\begin{Theorem}\label{cover}
Let $G$ be  an unmixed bipartite graph on the vertex set $V_n=\{x_1,\dots, x_n\}\\
\cup \{y_1,\dots, y_n\}$.
Suppose that $G$ has an induced subgraph $K_{m,m}$ {\em($m\geq 2$)}, with the vertex set  $V_m=\{x_{i_{1}},\ldots,x_{i_{m}}\}\cup\{y_{i_{1}},\ldots,y_{i_{m}}\}$. Let $H$ be the subgraph of $G$  induced by the subset $V_n\setminus V_{m-1}$, where $V_{m-1}=\{x_{i_{1}},\ldots,x_{i_{m-1}}\}\cup\{y_{i_{1}},\ldots,y_{i_{m-1}}\}$. Then there exists a one-to-one correspondence between the sets $\mathcal{M}(G)$, respectively $\mathcal{M}(H)$, of minimal vertex covers of $G$,
respectively $H$. More precisely, for all subsets $C\in V_n\setminus V_{m-1}$, we have
\begin{itemize}
\item[(i)] If $x_{i_m}\in  C$, then $C\in  \mathcal{M}(H)\Leftrightarrow C\cup \{x_{i_{1}},\ldots,x_{i_{m}}\}\in  \mathcal{M}(G)$;
\item[(ii)] If $x_{i_m}\notin  C$, then $C\in  \mathcal{M}(H)\Leftrightarrow C\cup \{y_{i_{1}},\ldots,y_{i_{m}}\}\in  \mathcal{M}(G)$.
\end{itemize}
\end{Theorem}
\begin{proof}
  $(\Rightarrow)$ Let  $x_{i_m}\in  C$, put $B=C\cup \{x_{i_{1}},\ldots,x_{i_{m}}\}$. We show that $B\in  \mathcal{M}(G)$.
 We distinguish into the following three cases:

(1) For any $\{x_k,y_\ell\}\in E(H)$, we have  $C\cap \{x_k,y_\ell\}\neq \emptyset$ since $C\in  \mathcal{M}(H)$. In particular, $B\cap \{x_k,y_\ell\}\neq \emptyset$.

(2) For any $\{x_{i_k},y_\ell\}\in E(G)$ with $k\in [m]$, we have $x_{i_k}\in B\cap \{x_{i_k},y_\ell\}$.

(3) For any  $\{x_{i_k},y_{i_\ell}\} \in E(G)$ with $\ell \in [m-1]$ and $k\notin [m]$. By Lemma \ref{unmixed} (ii), we obtain  that $\{x_{i_k},y_{i_m}\} \in E(G)$ since $\{x_{i_\ell},y_{i_m}\} \in E(G)$. Hence $\{x_{i_k},y_{i_m}\} \in E(H)$, it implies $C\cap \{x_{i_k},y_{i_m}\}\neq \emptyset$.

Note that $|C\cap \{x_{i_m},y_{i_m}\}|=1$, we have $y_{i_m}\notin C$ because of $x_{i_m}\in C$. It follows that $x_{i_k}\in C$. Thus
 $B\cap \{x_{i_k},y_{i_m}\}=\{x_{i_k}\}$ and $B\in  \mathcal{M}(G)$.

 If $x_{i_m}\notin  C$, then  $y_{i_m}\in  C$. Put $B=C\cup \{y_{i_{1}},\ldots,y_{i_{m}}\}$. By similar arguments as above, we can show that $B\in  \mathcal{M}(G)$.

$(\Leftarrow)$ Let $B=C\cup \{x_{i_{1}},\ldots,x_{i_{m}}\}$ and $B\in  \mathcal{M}(G)$. Then $|B\cap \{x_{i_{\ell}},y_{i_{\ell}}\}|=1$ for any $\ell\in [m-1]$. This implies that either $B\cap \{x_{i_{\ell}},y_{i_{\ell}}\}=\{x_{i_{\ell}}\}$ or $B\cap \{x_{i_{\ell}},y_{i_{\ell}}\}=\{y_{i_{\ell}}\}$.
If $B\cap \{x_{i_{\ell}},y_{i_{\ell}}\}=\{x_{i_{\ell}}\}$, then $y_{i_{\ell}}\notin B$, it means that $B\cap \{x_{i_{m}},y_{i_{\ell}}\}=\{x_{i_{m}}\}$ because of $\{x_{i_{m}},y_{i_{\ell}}\}\in E(G)$. Set $C=B\setminus \{x_{i_{1}},\ldots,x_{i_{m-1}}\}$. Then  $B\cap \{x_k ,y_l\}\neq \emptyset$ for any
$\{x_k ,y_l\} \in E(H)$, and  $B\cap \{x_k ,y_l\}\subset C$. Hence  $C$ is a vertex cover
of $H$. Since $|C|=|B|-(m-1)=n-m+1$, we get $C\in  \mathcal{M}(H)$. If $B\cap \{x_{i_{\ell}},y_{i_{\ell}}\}=\{y_{i_{\ell}}\}$,
then $x_{i_{\ell}}\notin B$ and $C=B\setminus \{y_{i_{1}},\ldots,y_{i_{m-1}}\}\in  \mathcal{M}(H)$.
\end{proof}

\medskip
 For every monomial $z^{a_1}\cdots z^{a_n}\in R$, we write $Z^{\mathbf a}$ for $z^{a_1}\cdots z^{a_n}$, where $\mathbf{a}=(a_1,\ldots,a_n)$ with  $a_i\geq 0$ for any $i\in [n]$. We call $\mathbf{a}$ to be the {\em exponential vector} of $z^{\mathbf a}$.
 Let $\mathcal{G}(I)$ denote the
unique minimal set of monomial generators of a monomial ideal.

\begin{Lemma}\label{rank}{\em(\cite[Lemma 10.3.19]{BH})} Let $I\subset R$ be a monomial ideal generated in a single degree  with  $\mathcal{G}(I)=\{z^{\mathbf {a_1}},\ldots,z^{\mathbf {a_m}}\}$, and let $A$ be the $m\times n$ matrix
whose rows are the vectors $\{\mathbf{a_1},\ldots,\mathbf{a_m}\}$. Then $\ell(I)=\text{rank}\,(A)$.
 \end{Lemma}

Now, we are ready to prove  major result of this section.
\begin{Theorem}\label{main}
Let $G$ be  a connected  unmixed bipartite graph on the vertex set $V_n=\{x_1,\dots, x_n\}\cup \{y_1,\dots, y_n\}$.
Suppose that $G$ has an induced subgraph $K_{m,m}$  {\em($m\geq 2$)} with the vertex set  $V_m=\{x_{i_{1}},\ldots,x_{i_{m}}\}\cup\{y_{i_{1}},\ldots,y_{i_{m}}\}$. Let $H$ be the induced subgraph of $G$ on the subset $V_n\setminus V_{m-1}$, where $V_{m-1}=\{x_{i_{1}},\ldots,x_{i_{m-1}}\}\cup\{y_{i_{1}},\ldots,y_{i_{m-1}}\}$. Let $H_{\mathcal{L}_G}$ {\em (resp. $H_{\mathcal{L}_H}$)} be the Hibi ideal of the lattice $\mathcal{L}_G$ {\em (resp. $\mathcal{L}_H$)}. Then $H_{\mathcal{L}_G}$ is Freiman if and only if $H_{\mathcal{L}_H}$ is Freiman.
\end{Theorem}
\begin{proof} Without loss of generality, we can assume that  $i_j=j$ for all $j\in [m]$. By Theorem \ref{cover}, we obtain that
  generators of $H_{\mathcal{L}_G}$ can be obtained from    generators of $H_{\mathcal{L}_H}$.
Therefore, all  generators of $H_{\mathcal{L}_G}$  are of the form
$$u_1\prod\limits_{i=1}^{m}x_i,\ldots, u_s\prod\limits_{i=1}^{m}x_i, v_1\prod\limits_{i=1}^{m}y_i,\ldots, v_t\prod\limits_{i=1}^{m}y_i $$
where $x_{m}u_1,\ldots,x_{m}u_s,y_{m}v_1,\ldots,y_{m}v_t$ are all  generators of $H_{\mathcal{L}_H}$ and monomials $u_{i},v_{j}\in K[x_{m+1},\ldots,
x_{n},y_{m+1},\ldots,y_{n}]$ for all $i\in [s]$, $j\in [t]$. It follows that $$\mu(H_{\mathcal{L}_G})=\mu(H_{\mathcal{L}_H})\eqno(1)$$

Now we prove that $\mu(H^2_{\mathcal{L}_G})=\mu(H^2_{\mathcal{L}_H})$.
Note that
\begin{eqnarray*}H_{\mathcal{L}_G}&=&(u_1\prod\limits_{i=1}^{m}x_i,\ldots, u_s\prod\limits_{i=1}^{m}x_i, v_1\prod\limits_{i=1}^{m}y_i,\ldots, v_t\prod\limits_{i=1}^{m}y_i)\\
&=&(u_1,\ldots, u_s)\prod\limits_{i=1}^{m}x_i+(v_1,\ldots, v_t)\prod\limits_{i=1}^{m}y_i
\end{eqnarray*}
and
\begin{eqnarray*}H_{\mathcal{L}_H}&=&(x_{m}u_1,\ldots,x_{m}u_s,y_{m}v_1,\ldots,y_{m}v_t)\\
&=&(u_1,\ldots,u_s)x_{m}+(v_1,\ldots,v_t)y_{m}.
\end{eqnarray*}
It follows that
\begin{eqnarray*}H^2_{\mathcal{L}_G}&=&[(u_1,\ldots, u_s)\prod\limits_{i=1}^{m}x_i+(v_1,\ldots, v_t)\prod\limits_{i=1}^{m}y_i]^2\\
&=&(u_1,\ldots, u_s)^2\prod\limits_{i=1}^{m}x_i^2+(v_1,\ldots, v_t)^2\prod\limits_{i=1}^{m}y_i^2+(u_1,\ldots, u_s)(v_1,\ldots, v_t)\prod\limits_{i=1}^{m}x_iy_i
\end{eqnarray*}
and
\begin{eqnarray*}
H^2_{\mathcal{L}_H}&=&[(u_1,\ldots,u_s)x_{m}+(v_1,\ldots,v_t)y_{m}]^2\\
&=&(u_1,\ldots,u_s)^2x_{m}^2+(v_1,\ldots,v_t)^2y_{m}^2+(u_1,\ldots,u_s)(v_1,\ldots,v_t)x_{m}y_{m}.
\end{eqnarray*}

Note that  generators of ideals $(u_1,\ldots, u_s)^2\prod\limits_{i=1}^{m}x_i^2$, $(v_1,\ldots, v_t)^2\prod\limits_{i=1}^{m}y_i^2$ and
$(u_1,\ldots, u_s)\\ \cdot(v_1,\ldots, v_t)\prod\limits_{i=1}^{m}x_iy_i$ are different from each other, and so are  generators of ideals
$(u_1,\ldots,u_s)^2x_{m}^2$, $(v_1,\ldots,v_t)^2y_{m}^2$ and $(u_1,\ldots,u_s)(v_1,\ldots,v_t)x_{m}y_{m}$.
Thus we get \begin{eqnarray*}
 \hspace{1.0cm}
\mu(H^2_{\mathcal{L}_G})&=&\mu((u_{1},\ldots,u_{s})^{2})+\mu((v_{1},\ldots,v_{t})^{2})+\mu((u_{1},\ldots,u_{s})(v_{1},\ldots,v_{t}))\\
&=&\mu(H^2_{\mathcal{L}_H}).  \hspace{9.3cm} (2)
\end{eqnarray*}

By (1) and (2), it is enough to show that $\ell(H_{\mathcal{L}_G})=\ell(H_{\mathcal{L}_H})=n-m+2$.

Since $G$ is a connected  unmixed bipartite graph and  $K_{m,m}$ is its  unique induced subgraph, we obtain that $H$ is a C-M bipartite graph on the  vertex set $\{x_{m},\dots, x_n\}\cup \{y_{m},\dots, y_n\}$ by Theorem \ref{complete} and $H$ is connected. Indeed, there is some  edge in $G$ connects a vertex of $V_m$ and a vertex of $V_n\setminus V_m$. Say $x_{k}y_{\ell}\in E(G)$ with $k\in [n]\setminus [m]$ and $\ell\in [m]$. One has $x_{k}y_{m}\in E(H)$. Actually, if $\ell\in [m-1]$, then the desired result from Lemma \ref{unmixed} because of $x_{k}y_{\ell}\in E(H)$ and $x_{\ell}y_{m}\in E(H)$.
Note that   $\{y_{m},\ldots,y_{n}\}$ and $\{x_{m},\ldots,x_{j},y_{j+1},\ldots,y_{n}\}$ are the minimal vertex covers of $H$ for all $m\leq j\leq n$. It follows that
$$\{x_{m}\cdots x_{n},\,x_{m}\cdots x_{n-1}y_{n},\,\ldots,\,x_{m}y_{m+1}\cdots y_{n},\,y_{m}\cdots y_{n}\}\subset \mathcal{G}(H_{\mathcal{L}_H}).$$
Let $\mathcal{G}(H_{\mathcal{L}_H})=\{X^{\mathbf{b}_{1}},\ldots,X^{\mathbf{b}_{q}}\}$, where
$\mathbf{b}_{j}=(b_{mj},\ldots,b_{nj},1-b_{mj},\ldots,1-b_{nj})$,  $b_{ij}\in \{0,1\}$, $m\leq i\leq n$ and $j\in [q]$, is the exponential
vector of  monomial $X^{\mathbf{b}_{j}}=x_{m}^{b_{mj}}\cdots x_{n}^{b_{nj}}y_{m}^{1-b_{mj}}\cdots y_{n}^{1-b_{nj}}$, and
let $A$ be the $q\times 2(n-m+1)$ matrix whose rows are the vectors $\{\mathbf{b_1},\ldots,\mathbf{b_q}\}$.
Let $\mathbf{\alpha}_{0}=(\underbrace {1,\ldots,1}_{n-m+1},0,\ldots,0)$,  $\mathbf{\alpha}_{n-m+1}=(0,\ldots,0,\underbrace {1,\ldots,1}_{n-m+1})$,
$\mathbf{\alpha}_{i}=(\underbrace {1,\ldots,1}_{n-m+1-i},0,\ldots,0,\underbrace {1,\ldots,1}_i)\in \mathbb{Z}^{2(n-m+1)}$ for any $i\in [n-m]$,
then $\{\mathbf{\alpha}_{0},\mathbf{\alpha}_{1},\ldots,\mathbf{\alpha}_{n-m+1}\}\subset \{\mathbf{b_1},\ldots,\mathbf{b_q}\}$.
 Let $B$ be a matrix whose rows are the vectors $\{\alpha_{0},\alpha_{1},\ldots,\alpha_{n-m+1}\}$.
 Then $\text{rank}\,(B)=n-m+2$. This implies that $\text{rank}\,(A)\geq n-m+2$.

On the other hand, for any  $\mathbf{b}_{j}\in \{\mathbf{b_1},\ldots,\mathbf{b_q}\}$, we have
$$\mathbf{b}_{j}=b_{nj}\alpha_{0}+(b_{n-1,j}-b_{nj})\alpha_{1}+\cdots+(b_{mj}-b_{m+1,j})\alpha_{n-m}+(1-b_{mj})\alpha_{n-m+1}.$$
Therefore, $\text{rank}\,(A)=n-m+2$.

Let $\mathcal{G}(H_{\mathcal{L}_G})=\{X^{\mathbf{a}_{1}},\ldots,X^{\mathbf{a}_{q}}\}$, where
$\mathbf{a}_{j}=(a_{1j},\ldots,a_{nj},1-a_{1j},\ldots,1-a_{nj})$,  $a_{ij}\in \{0,1\}$, $i\in [n]$ and $j\in [q]$, is the exponential
vector of  $X^{\mathbf{a}_{j}}=x_{1}^{a_{1j}}\cdots x_{n}^{a_{nj}}y_{1}^{1-a_{1j}}\\
\cdot \cdots y_{n}^{1-a_{nj}}$, and let $D$ be the $q\times 2n$ matrix whose rows are the vectors $\{\mathbf{a_1},\ldots,\mathbf{a_q}\}$.
For convenience, we may assume that  $x_m$ is a factor of $X^{\mathbf{b}_{i}}$ for any $1\leq i\leq t$ and $y_m$ is a factor of $X^{\mathbf{b}_{i}}$ for any $t+1\leq i\leq n$.
By the relationship between  generators of ideal $H_{\mathcal{L}_G}$ and $H_{\mathcal{L}_H}$, we obtain that
$\prod\limits_{i=1}^{m}x_i$ is a factor of $X^{\mathbf{a}_{i}}$ for any $1\leq i\leq t$ and $\prod\limits_{i=1}^{m}y_i$ is a factor of $X^{\mathbf{a}_{i}}$ for any $t+1\leq i\leq n$. Thus the $i$-th column of matrix $D$ is a copy of its $m$-th column for $1\leq i\leq m-1$  and
the $j$-th column of matrix $D$ is a copy of its $(n+m)$-th column for $n+1\leq j\leq n+m-1$.
Therefore, $\text{rank}\,(D)=\text{rank}\,(A)$.
The desired result from Lemma \ref{rank}.
\end{proof}

\begin{Theorem}\label{multi}
Let $G$ be  a connected  unmixed bipartite graph on the vertex set $V_n=\{x_1,\dots, x_n\}\cup \{y_1,\dots, y_n\}$. Let $K_{m_{1},m_{1}},\ldots,K_{m_{s},m_{s}}$
be all induced complete bipartite subgraphs of $G$ on the subsets $V_{m_1},\ldots,V_{m_s}$ of $V(G)$ respectively, where $m_i\geq 2$ for any $i\in [s]$ and $s\geq 1$. Let $V_{m_i}=\{x_{i1},\ldots, x_{im_i}\}\cup \{y_{i1},\ldots, y_{im_i}\}$ and $H$ the induced subgraph of $G$  on the subset $V_n\setminus (\bigcup\limits_{i=1}^s V_{m_i-1})$. Let $H_{\mathcal{L}_G}$ {\em (resp. $H_{\mathcal{L}_H}$)} be the Hibi ideal of the lattice $\mathcal{L}_G$ {\em (resp. $\mathcal{L}_H$)}. Then $H_{\mathcal{L}_G}$ is Freiman if and only if $H_{\mathcal{L}_H}$ is Freiman.
\end{Theorem}
\begin{proof} We apply induction on $s$.  The case $s=1$ follows from Theorem \ref{main}. Now we
assume that $s\geq 2$. Let $H'$ be the induced subgraph of $G$  on the subset $V_n\setminus  V_{m_s-1}$. Then $H_{\mathcal{L}_G}$ is Freiman if and only if $H_{\mathcal{L}_{H'}}$ is Freiman by Theorem \ref{main}. Note that $H'$ is a
connected unmixed bipartite graph  by Theorem \ref{cover} and similar arguments in the proof of Theorem \ref{main}.
Thus, by induction hypothesis,  $H_{\mathcal{L}_{H'}}$ is Freiman  if and only if $H_{\mathcal{L}_H}$ is Freiman.
 The proof is completed.
 \end{proof}

\medskip
Two typical examples of Freiman  cover ideals of unmixed bipartite graphs  as in
Theorem \ref{multi} are shown in Figure  $2$ and Figure  $3$.

\begin{center}
\setlength{\unitlength}{1.5mm}
\begin{picture}(120,35)
\linethickness{1pt}
\thicklines
\multiput(5,10)(10,0){5}{\circle*{1}}
\multiput(5,30)(10,0){5}{\circle*{1}}
\multiput(5,10)(10,0){5}{\line(0,1){20}}
\multiput(5,30)(10,0){4}{\line(1,-2){10}}
\multiput(5,30)(10,0){3}{\line(1,-1){20}}
\multiput(5,30)(10,0){2}{\line(3,-2){30}}
\multiput(5,10)(20,0){2}{\line(1,2){10}}
\put(5,30){\line(2,-1){40}}
\put(4,32){$x_5$}
\put(14,32){$x_1$}
\put(24,32){$x_2$}
\put(34,32){$x_4$}
\put(44,32){$x_3$}
\put(4,6){$y_5$}
\put(14,6){$y_1$}
\put(24,6){$y_2$}
\put(34,6){$y_4$}
\put(44,6){$y_3$}
\put(19,-1){$Figure\ \ 2$}

\multiput(55,10)(10,0){5}{\circle*{1}}
\multiput(55,30)(10,0){5}{\circle*{1}}
\multiput(55,10)(10,0){5}{\line(0,1){20}}
\multiput(55,30)(10,0){4}{\line(1,-2){10}}
\multiput(55,30)(10,0){3}{\line(1,-1){20}}
\multiput(55,30)(10,0){2}{\line(3,-2){30}}
\multiput(65,10)(10,0){2}{\line(1,2){10}}
\put(55,30){\line(2,-1){40}}
\put(65,10){\line(1,1){20}}
\put(54,32){$x_1$}
\put(64,32){$x_4$}
\put(74,32){$x_2$}
\put(84,32){$x_5$}
\put(94,32){$x_3$}
\put(54,6){$y_1$}
\put(64,6){$y_4$}
\put(74,6){$y_2$}
\put(84,6){$y_5$}
\put(94,6){$y_3$}
\put(69,-1){$Figure\ \  3$}
\end{picture}
\end{center}

\vspace{0.5cm}

\medskip
For the convenience of the following description, we call the graph $G$ in theorem 1.6 {\em almost} complete C-M bipartite graph.
An immediate consequence of the above theorem and Theorem \label{zhao} is the following corollary.
\begin{Corollary}\label{good}
Let $G$ be  a  connected unmixed bipartite graph on the vertex set $V_n=\{x_1,\dots, x_n\}\cup \{y_1,\dots, y_n\}$. Then  $H_{\mathcal{L}_{G}}$ is Freiman  if and only if  $G$ is an almost  complete C-M bipartite graph,   or
there exist some  complete bipartite subgraphs $K_{m_{1},m_{1}},\ldots,K_{m_{s},m_{s}}$ of $G$ such that the induced graph of $G$ on
$V_n\setminus (\bigcup\limits_{i=1}^s V_{m_i-1})$ is an almost  complete C-M bipartite graph, where
 $V_{m_i}=\{x_{i1},\ldots, x_{im_i}\}\cup \{y_{i1},\ldots, y_{im_i}\}$ is  the vertex set of $K_{m_{i},m_{i}}$ for any $i\in [s]$.
\end{Corollary}

\medskip
The following example shows that the cover ideal  of the union of two  disjoint graphs, whose cover  ideals are  of Freiman, is not always Freiman ideal.
\begin{Example}
Let $H_{\mathcal{L}_{G_1}}=(x_1x_2,x_1y_2,y_1y_2)$, $H_{\mathcal{L}_{G_2}}=(x_3x_4,x_3y_4,y_3y_4)$ and $H_{\mathcal{L}_{G_1\sqcup G_2}}\\
=H_{\mathcal{L}_{G_1}}H_{\mathcal{L}_{G_2}}$
be cover ideals of graphs $G_1$, $G_2$ and $G_1\sqcup G_2$ in Figure $4$, respectively. By using CoCoA, we obtain that
$H^2_{\mathcal{L}_{G_1}}=(x_1^2x_2^2,x_1^2x_2y_2,x_1^2y_2^2,x_1x_2y_1y_2,x_1y_1y_2^2,y_1^2y_2^2)$, $H^2_{\mathcal{L}_{G_2}}=(x_3^2x_4^2,x_3^2x_4y_4,x_3^2y_4^2,x_3x_4y_3y_4,x_3y_3y_4^2,y_3^2y_4^2)$,
$H^2_{\mathcal{L}_{G_1\sqcup G_2}}=H^2_{\mathcal{L}_{G_1}}H^2_{\mathcal{L}_{G_2}}$, $\ell(H_{\mathcal{L}_{G_1}})=\ell(H_{\mathcal{L}_{G_2}})=3$
and $\ell(H_{\mathcal{L}_{G_1\sqcup G_2}})=5$. Thus $\mu(H^2_{\mathcal{L}_{G_i}})-\ell(H_{\mathcal{L}_{G_i}})\mu(H_{\mathcal{L}_{G_i}})+{\ell(H_{\mathcal{L}_{G_i}})\choose 2}=6-3\times 3+3=0$ for $i=1,2$
and $\mu(H^2_{\mathcal{L}_{G_1\sqcup G_2}})-\ell(H_{\mathcal{L}_{G_1\cup G_2}})\mu(H_{\mathcal{L}_{G_1\sqcup G_2}})+{\ell(H_{\mathcal{L}_{G_1\sqcup G_2}})\choose 2}=36-5\times 9+10=1$. Therefore, the cover ideals of graphs $G_1$ and $G_2$ are Freiman, but  the cover ideal of their disjoint union $G_1\sqcup G_2$ is not Freiman.
\end{Example}

\medskip
\hspace{-6mm} {\bf Acknowledgments}

 \vspace{3mm}
\hspace{-6mm}  This research is supported by the National Natural Science Foundation of China (No.11271275) and  by foundation of the Priority Academic Program Development of Jiangsu Higher Education Institutions.

%\medskip

%\end{document}
\end{document}